\documentclass[12pt,a4paper,reqno]{amsart}

\usepackage{amsmath,amsfonts,amsthm,amsopn,cite,mathrsfs}
\usepackage{epsfig,verbatim}
\usepackage{subfigure}

\newcommand{\tfa}{time-frequency analysis}

\newcommand{\stft}{short-time Fourier transform}

\newcommand{\tf}{time-frequency}

\newcommand{\fif}{if and only if}
\newcommand{\tfs}{time-frequency shift}

\newcommand{\onb}{orthonormal basis}

\newcommand{\modsp}{modulation space}
\newcommand{\psdo}{pseudodifferential operator}

\newtheorem{tm}{Theorem}[section]
\newtheorem{lemma}[tm]{Lemma}
\newtheorem{prop}[tm]{Proposition}
\newtheorem{cor}[tm]{Corollary}

\newcommand{\rem}{\noindent\textsl{REMARK:}}

 \theoremstyle{definition}
 \newtheorem{definition}{Definition}

\newcommand{\beqa}{\begin{eqnarray*}}
\newcommand{\eeqa}{\end{eqnarray*}}

\newcommand{\field}[1]{\mathbb{#1}}
\newcommand{\bR}{\field{R}}        
\newcommand{\bN}{\field{N}}        
\newcommand{\bZ}{\field{Z}}        
\newcommand{\bC}{\field{C}}        
\newcommand{\vf}{\varphi}

 \def\cF{\mathcal{F}}              
 \def\cS{\mathcal{S}}
 
 \def\cH{\mathcal{H}}
 \def\cB{\mathcal{B}}
 
 \def\cG{\mathcal{G}}
 \def\cM{\mathcal{M}}

 \def\cC{\mathcal{C}}
 
 \def\cO{\mathcal{O}}

\def\rd{\bR^d}
\def\cd{\bC^d}

\def\rdd{{\bR^{2d}}}
\def\zdd{{\bZ^{2d}}}

\def\lrd{L^2(\rd)}

\def\mvv{M_v^1}
\def\mif{M^{\infty , 1}}

\def\Lpm{L ^p_m}
\def\Lpqm{L ^{p,q}_m}

\def\Mpqm{M_m^{p,q}}

\def\intrd{\int_{\rd}}
\def\intrdd{\int_{\rdd}}

\def\<{\left<}
\def\>{\right>}

\def\inv{^{-1}}

\def\mv1{M_v^1}

\newcommand{\vs}{\vspace{3 mm}}

\newcommand{\jjj}{J_{1/\theta }}
\newcommand{\Mmpq}{M^{p,q}_m}
\newcommand{\Mmpqd}{M^{p,q}_m(\rd )}
\newcommand{\scal}[2]{\langle #1,#2\rangle}
\newcommand{\rr}[1]{\mathbb R^{#1}}
\newcommand{\cc}[1]{\mathbb C^{#1}}

\newcommand{\fy}{\varphi}

\begin{document}

\begin{abstract}
We study the range of  \tf\ localization operators acting on \modsp s
and prove a lifting theorem. As an application we also characterize
the range of Gabor multipliers, and, in the realm of complex analysis,
we characterize the range of certain Toeplitz operators on weighted
Bargmann-Fock spaces.  The main tools are the construction of
canonical isomorphisms between \modsp s of Hilbert-type and a refined
version of the spectral invariance of \psdo s. On the technical level
we prove a new class of inequalities for weighted gamma functions.  
\end{abstract}

\title[The Range of Localization Operators]{The Range of
Localization Operators  and\\ Lifting Theorems for  Modulation and Bargmann-Fock Spaces}

\author{Karlheinz Gr\"ochenig}

\address{Faculty of Mathematics \\
University of Vienna \\
Nordbergstrasse 15 \\
A-1090 Vienna, Austria}

\email{karlheinz.groechenig@univie.ac.at}

\author{Joachim Toft}

\address{Department of Computer science, Physics and Mathematics,
Linn{\ae}us University, V{\"a}xj{\"o}, Sweden}

\email{joachim.toft@lnu.se}

\subjclass[2000]{}
\date{}
\thanks{ K.\ G.\ was
  supported in part by the  project P2276-N13  of the
Austrian Science Foundation (FWF)} 
\maketitle

\par

\section{Introduction}

\par

The precise description of the range of a linear operator is usually
difficult, if not impossible, because this amounts to a
characterization of which operator equations are solvable. 
In this paper we study the range of an important class of \psdo s,
so-called \tf\ localization operators, and we prove an isomorphism
theorem between \modsp s with respect to different weights. 

The guiding example to develop an intuition for  our  results 
is the class of multiplication operators.  Let $m\geq 0$ be a weight
function on $\rr d$ and define the weighted space
$L^p_m(\rr d)$ by the norm $\|f\|_{\Lpm} = \Big(\intrd |f(x)|^p m(x)^p \,
dx\Big)^{1/p} = \|f m \|_{L^p}$.
%
%
Let $\cM _a $ be the multiplication
operator defined by $\cM _a f  = a f$. Then $\Lpm $ is precisely the
range of the multiplication operator $\cM_{1/m}$.

We will prove an similar result for \tf\ localization operators
between weighted \modsp s. To set up terminology, let 
$\pi (z) g (t)  = e^{2\pi i \xi\cdot t} g(t-x)$  denote the \tfs\ by
$z= (x,\xi ) \in \rdd $ acting on a function $g$ on $\rd $. The
corresponding transform is the \stft\ of a function defined by 
$$
V_g f(z) = \intrd f(t) \overline{g(t-x)}\, e^{2\pi i \xi \cdot t} \,
dt = \langle f , \pi (z) g\rangle  \, . 
$$
The standard function spaces of \tfa\ are the \modsp s. The \modsp\
norms measure smoothness in the \tf\ space (phase space in the
language of physics) by imposing a norm on the \stft\ of a function
$f$. As a special case we mention the \modsp s $M^p_m(\rd )$ for $1\leq
p\leq \infty $ and a non-negative  weight function $m$. Let
$$
{h} (t) = 2^{d/4}e^{-\pi t^2}=2^{d/4}e^{-\pi t\cdot t},\qquad t\in \rd
$$
denote the (normalized) Gaussian. Then the \modsp\
$M^p_m (\rd )$ is defined by the norm  
$$
\|f\|_{M^p_m } = \|V_{{h} } f \|_{\Lpm }  \, .
$$
The localization operator $A^g_m$ with respect to the ``window'' $g$,
usually some test function, and the symbol or multiplier $m$ is
defined formally by the integral
$$
 A_m^{g}f = \int_{\mathbb{R}^{2d}} m(z) V_g f(z
)\pi (z) g \,  dz \, . 
 $$
Localization operators constitute an important class of \psdo s and
occur under different names such as Toeplitz operators or anti-Wick
operators. They were introduced by Berezin as a form of
quantization~\cite{Berezin71},  and are nowadays  applied in mathematical signal
processing for \tf\ masking of signals and for phase-space
localization~\cite{daub88}.  An equivalent form occurs in complex
analysis as Toeplitz operators on Bargmann-Fock
space~\cite{Cob01,BCo87,BCo94}. 
In hard analysis they are used to approximate \psdo s and in some
proofs of the sharp G{\aa}rding inequality and the Fefferman-Phong
inequality~\cite{Toft00,lerner03,LM06}. For the analysis of localization operators
with \tf\ methods we refer to \cite{CG03} and the references given
there, for a more analytic point of view we recommend~\cite{Toft04,Toft07}. 

\par

A special case of our main result can be formulated as follows. 
By an isomorphism between two Banach spaces $X$ and $Y$ we understand
a bounded and invertible operator  from $X$ onto $Y$. 

\par

\begin{tm} \label{basic}
  Let $m$ be a non-negative continuous  symbol on $\rdd $ satisfying $m(w +z)
\leq e^{a |w|} m(z)$ for $w, z \in \rdd $ and assume that $m$ is
radial in each \tf\ coordinate. Then for suitable test functions $g$
the localization operator $A^g _m$ is an isomorphism from the \modsp\
$M^{p}_\mu (\rd )  $ onto $M^{p} _{\mu / m} (\rd ) $ for all $1\leq p \leq
\infty $ and moderate weights $\mu $. 
\end{tm}

\par

We see that the  range of  a localization operator 
exhibits the same behavior as the multiplication operators. For the
precise formulation with all assumptions stated we refer to
Section~4. The above isomorphism theorem can also be intepreted as a
lifting theorem, quite in  analogy  with the lifting property of
Besov spaces~\cite{triebel83}.  However, whereas Besov spaces $\dot{B}^{p,q}_s$
with different
smoothness $s$ are isomorphic via Fourier multipliers, the case of
\modsp s is  more subtle because in this case the \tf\
smoothness is parametrized by a weight function on $\rdd $ rather than
by a single number$s \in \bR $. The lifting operators between \modsp s are precisely
the localization operators with the weight $m$. 

\par

The isomorphism theorem for localization operators stated above is
preceded by many contributions, which were already listed in
~\cite{GT10}. In particular, in~\cite{CG06} it was  shown  that for
moderate symbol functions with 
subexponential growth the localization operator is a Fredholm
operator, i.e., it differs from an invertible operator only by a
finite-rank operator. Here we show that such operators are even
isomorphisms between the corresponding \modsp s, even under weaker
conditions than needed for the Fredholm property.  In the companion
paper~\cite{GT10} we proved the isomorphism property for weight
functions of \emph{polynomial type}, i.e.,
\begin{equation}\label{poltype}
c (1+|z|)^{-N} \leq m(z) \leq C
((1+|z|)^{N}.
\end{equation}

\par

The contribution of Theorem~\ref{basic} is the extension of the class
of possible weights. In particular, the isomorphism property holds
also for subexponential weight functions of the form $m(z) =
e^{a|z|^b}$  or $m(z) = e^{a|z|/\log (e+|z|)}$ for $a>0, 0<b<1$, which
are often considered in \tfa . We remark that this generality comes at
the price of imposing the  radial symmetry on the weight $m$.
Therefore, our  results  are not applicable to  all situations covered
by \cite{GT10}, where no radial symmetry is required.  

\par
 
Although the extension to weights of ultra-rapid  growth looks like a
routine generalization, it is not. The proof of Theorem~\ref{basic}
for weights of polynomial type is based on  a deep theorem of Bony and
Chemin~\cite{BC94}.  They  construct a one-parameter group of isomorphisms from
$\lrd $ onto the \modsp s $M^{2}_{m^t}(\rd )$ for $t\in  \bR
$. Unfortunately, the pseudodifferential calculus developed in 
\cite{BC94} requires polynomial growth conditions, and,  to our
knowledge, an   extention
to symbols of  faster growth is  not available.

On a technical level, 
our main contribution  is the construction of  canonical
isomorphisms  between $\lrd $ and the \modsp s $M^2_m(\rd )$  of Hilbert
type.  In fact, such an isomorphism is given by a \tf\ localization
operator with Gaussian window ${h} (t) = 2^{d/4}e^{-\pi t^2}$. 

\par

\begin{tm} \label{canon0}
Assume that $m$ is a continuous  moderate weight function of at most
exponential growth and radial in each \tf\ coordinate. Then the
localization operator $A^{h} _m$ is an isomorphism from $\lrd $ onto
$M^2_{1/m} (\rd ) $. 
\end{tm}

\par

This theorem replaces the result of Bony and Chemin. The main point is
that Theorem~\ref{canon0} also  covers symbols of ultra-rapid growth.  Its proof requires most
of our efforts. We take a \tf\ approach rather than using classical
methods from pseudodifferential calculus. 
In the course of its proof we will establish  new
 inequalities for weighted gamma functions of the form
$$
C\inv \leq    \int _0^\infty \theta (\sqrt{x/\pi}) \frac{x^n}{n!} e^{-x} dx
  \, \int _0^\infty \frac{1}{\theta (\sqrt{x/\pi})}
  \frac{x^n}{n!} e^{-x} dx \leq C  \text{ for all } n\in \bN \cup
  \{0\} 
 \, .
$$

The proof method for Theorem~\ref{basic} may be of interest in
itself. 
Once the canonical isomorphisms are in place (Theorem \ref{canon0}), the proof of
Theorem~\ref{basic} proceeds as follows. It is easy to establish that
$A_m$ is an isomorphism from $M^2_{\sqrt m}$ to $M^2_{1/\sqrt m}$,
so  that 
the composition $A^g_{1/m} A^g_m$ is an isomorphism on
$M^2_{\sqrt{m}}$. Using the canonical isomorphisms of
Theorem~\ref{canon0}, one shows next that the operator $ V= A^{h} _{\sqrt{m}}
A^g_{1/m} A^g_m A^{h} _{1/\sqrt{m}}$ is an isomorphism on $L^2 $ and
that the (Weyl) symbol of this operator belongs to a generalized
Sj\"ostrand class. After these technicalities we apply the machinery
of spectral invariance of \psdo s  from ~\cite{gro06} to conclude that 
$V$ is invertible on all \modsp s $M^p_\mu $ (with $\mu $ compatible
with the conditions on $m$ and the window $g$). Since $V$ is a
composition of three isomorphisms, we then deduce that $A^g_m$ is an
isomorphism from $M^p_\mu $ onto $M^p_{\mu /m}$. 

\par

As an application  we prove (i)   a new isomorphism theorem for so-called
Gabor multipliers, which  are a discrete version of \tf\ localization
operators, and (ii) an isomorphism theorem for Toeplitz operators
between weighted Bargmann-Fock spaces of entire functions. To
formulate this result more explicitly, for  a non-negative weight
function $\mu$ and $1\leq p\leq \infty $,  let $\cF ^{p}_\mu (\cc d)$ be the
space of entire functions of $d$ complex variables defined by the norm
$$
\|F\|_{\cF^{p}_\mu}^p =\int_{\cd }    |F(z)|^p \mu(z)^p
 e^{-p\pi |z|^2/2} dz   <\infty \, ,
$$
and let $P$ be the usual projection from $L^1_{\mathrm{loc}}
(\mathbb{C}^d)$ to entire functions on $\cd $. The Toeplitz operator
with symbol $m$ acting on a function $F$  is defined to  be 
$
T_mF = P(mF)$. Then we show that  \emph{ the Toeplitz
operator $T_m$ is an isomorphism  from $\cF ^{p,q}_{\mu }(\cc d)$ onto $\cF
^{p,q}_{\mu /m}(\cc d)$ for every $1\leq p,q\leq \infty $ and every  moderate weight
$\mu $. }

The paper is organized as follows: In Section~2 we provide the precise
definition of \modsp s and  localization operators, and we  collect their
basic properties. In particular, we investigate the Weyl symbol
of the composition of two localization operators. In Section~3, which
contains our main contribution, we construct the canonical
isomorphisms and prove Theorem~\ref{canon0}. In Section~4 we derive a
refinement of the spectral invariance of \psdo s and prove the general
isomorphism theorem, of which Theorem~\ref{basic} is a special
case. Finally in Section~5 we give applications to Gabor multipliers
and Toeplitz operators. 

\par

\section{Time-Frequency Analysis and Localization Operators}

\par

We first set up the vocabulary of \tfa . For  the notation 
we follow the book \cite{book}. 
For  a point $z=(x,\xi )\in \rdd $ in phase space  the
\tfs\ of a function $f$ is $\pi (z)f(t) = e^{2\pi i \xi \cdot t}
f(t-x)$, $t\in \rd $. 

\vs 
\textbf{The \stft{:}} Fix a non-zero function $g\in L^1_{loc}(\rr d)$ which is usually
taken in a suitable space of Schwartz functions. Then the \emph{\stft} of a
function or distribution $f$ on $\rr d$ is defined to be
\begin{equation}
  \label{eq:can13}
V_gf (z) = \langle f, \pi (z) g\rangle  \qquad z\in \rdd \, ,  
\end{equation}
provided the scalar product is well-defined for every $z\in \rr
{2d}$. Here $g$ is called a ``window function''.  
If $f\in L^p(\rr d) $ and $g\in L^{p'}(\rr d)$ for the conjugate
parameter $p'=p/(p-1)$, then the \stft\ can
be written in integral form as 
$$
V_gf(z) = \intrd f(t) \overline{g(t-x)}\, e^{-2\pi i \xi \cdot t} \,
dt .
$$
In general, the  bracket $\langle \cdot , \cdot \rangle $ extends the inner product
on $\lrd $ to any dual pairing between a distribution space and its
space of test functions, for instance $g\in \cS (\rd )$ and $f\in \cS
' (\rd )$, but \tfa\ often needs larger distribution spaces. 

\par

\vs 

\textbf{Weight functions}: We call a locally bounded, strictly
positive  weight function $m$ on $\rdd $ \emph{moderate}, if  
$$
\sup _{z\in \rdd}\left ( \frac{m(z+y)}{m(z)},\frac{m(z-y)}{m(z)} \right ) :=  v(y) <\infty  \qquad \text{
  for all } y\in \rdd \, .
$$
The resulting function $v$ is a submultiplicative weight function,
i.e., $v$ is even and satisfies $v(z_1+z_2) \leq v(z_1) v(z_2)$ for all $z_1, z_2 \in \rdd $,
and then $m$ satisfies 
\begin{equation}
  \label{eq:4}
 m(z_1+z_2) \leq v(z_1) m(z_2) \qquad \text{ for all } z_1,z_2 \in
 \rdd \, . 
\end{equation}

\par

Given a submultiplicative weight function $v$ on $\rdd $, any weight
satisfying the condition~\eqref{eq:4} is called $v$-moderate. For a
fixed submultiplicative function $v$ the set
$$
\cM _v := \{m \in L^\infty _{loc}( \rdd )\,   :\,  0<  m(z_1+z_2) \leq v(z_1) m(z_2)\,\, \forall z_1, z_2 \in
\rdd \}
$$
 contains all $v$-moderate weights. 

\par

We will use several times that every $v$-moderate  weight $m \in \cM _v$
satisfies the following bounds:
\begin{equation}
  \label{eq:25}
  \frac{1}{v(z_1-z_2)} \leq \frac{m (z_1)}{m (z_2)} \leq
  v(z_1 - z_2)  \qquad \text{ for all } z_1,z_2 \in \rdd \, .
\end{equation}
This follows from \eqref{eq:4} by replacing $z_1$ with $z_1-z_2$. 


\vs 

\textbf{Modulation spaces} $\Mpqm $ for arbitrary weights: For the general
definition of \modsp s we choose the Gaussian function ${h} (t) =
2^{d/4} e^{-\pi t\cdot t}$ as the canonical window function. Then the
\stft\ is defined for arbitrary elements in the Gelfand-Shilov
space $(S^{1/2}_{1/2})'(\rr d)$ of generalized functions. 
The \modsp\ $\Mmpq (\rd ) $, $1\leq p,q <\infty $ consists of all elements $f\in (S^{1/2}_{1/2})'(\rr d)$ such that the norm 
\begin{equation}
  \label{eq:can16}
  \|f\|_{\Mmpq } = \Big(\intrd \Big( \intrd |V_{h} f(x,\xi )|^p
  m(x,\xi )^p \, dx\Big)^{q/p} \, d\xi \Big)^{1/q} = \|V_{{h} }
  f\|_{\Lpqm } \,         
\end{equation}
is finite. If $p=\infty $ or $q=\infty $, we make the usual
modification and replace the integral by the supremum norm $\| \cdot
\|_{L^\infty}$. If $m=1$, then we usually write $M^{p,q}$ instead of $M^{p,q}_m$. We also set $M^{p}_m=M^{p,p}_m$ and $M^{p}=M^{p,p}$.

\par

The reader who does not like general distribution spaces,
may interprete $\Mmpqd $ as the completion of the finite linear
combinations  of \tfs s $\cH _0 = \mathrm{span}\, \{ \pi (z) {h} :
z\in \rdd \}$ with respect to the $\Mmpq $-norm for $1\leq p,q
<\infty$ and as a weak$^*$-closure, when $p=\infty $ or $q=\infty $. 
These issues arise only for extremely rapidly decaying weight
functions. If $m\geq 1$ and $1\leq p,q \leq 2$, then $\Mmpqd$ is in
fact  a
subspace of $\lrd $. If $m$ is of polynomial type (cf. \eqref{poltype}), then $\Mmpqd $ is
a subspace of tempered distributions. This is the case that is usually
considered, although the theory of \modsp s was developed from the
beginning to include arbitrary moderate weight functions~\cite{fg89jfa,book,gro07c}. 

\par

\vs

\textbf{Norm equivalence}: Definition~\eqref{eq:can16} uses the Gauss function
as the canonical window. The definition of \modsp s, however, does not
depend on the particular choice of the window.
More precisely, if $g\in \mvv $, $g\neq 0$, and  $m\in \cM _v$, then
there exist constants $A,B>0$ such that 
\begin{equation}
  \label{eq:18}
A\, \|f\|_{\Mmpq } \leq \|V_gf\|_{\Lpqm } \leq B \|f\|_{\Mmpq } = B \|V_{h} f\|_{\Lpqm }  \, .  
\end{equation}
We will usually write 
$$
\|V_gf\|_{\Lpqm } \asymp \|f\|_{\Mmpq } 
$$
for the equivalent norms.

\par
\vs

\textbf{Localization operators:} Given a non-zero  window function $g\in M^1_v(\rd )$ and a
symbol or multiplier $m$ on $\rr {2d}$, the localization operator $A_m^{g}$ is
defined informally by
\begin{equation}
  \label{eq:1}
 A_m^{g}f = \int_{\mathbb{R}^{2d}} m(z) V_g f(z
)\pi (z) g \,  dz \, , 
  \end{equation}
provided the integral exists.
  A useful alternative  definition of $A^g_m$ is the weak definition 
\begin{equation}
  \label{eq:f1}
\langle A_m^g  f,k\rangle _{L^2(\rd )}  = \langle m
  {V}_{g} f,  {V}_{g} k\rangle _{L^2(\rdd
    )} \, .
\end{equation}
While in general the symbol $m$ may be a distribution in a \modsp\ of
the form $M^\infty _{1/v}(\rdd )$~\cite{CG03,Toft07}, we will investigate
only localization operators whose  symbol is a
moderate weight function. 

\par

Taking the \stft\ of \eqref{eq:1}, we find that 
$$
V_g(A^g_m f)(w) = \int _{\rdd } m(z) V_gf(z) \langle \pi (z) g, \pi
(w)g\rangle \, dz = \big( (m V_gf) \,  \natural \, V_gg \big)(w) \, ,
$$
with the usual  twisted convolution $\natural$ defined by 
$$
( F \, \natural \, G)(w) = \intrdd F(z) G(w-z) e^{2\pi i z_1 \cdot (z_2-w_2) }\,  dz .
$$

\par

Since
$$
F\mapsto \int _{\rdd }  F(z) \langle \pi (z) g, \pi
(\cdot )g\rangle \, dz
$$
is the projection from arbitrary tempered distributions on $\rdd
$ onto functions of the form $V_g f$ for some distribution $f$, the
localization operator can been seen as the composition of a
multiplication operator and the projection onto the space of \stft
s. In this light, localization operators resemble the classical
Toeplitz operators, which are multiplication operators following by a
projection onto analytic functions. Therefore they  are sometimes called Toeplitz
operators~\cite{GT10,Toft01,Toft07}. If the window $g$ is chosen to be 
the Gaussian, then this formal similarity can be made more
precise. See Proposition~\ref{p21}. 

\subsection{Mapping Properties of Localization Operators}

The mapping properties of localization operators on \modsp s resemble
closely the mapping properties of multiplication operators between
weighted $L^p$-spaces. The boundedness of localization operators has  been
investigated on many levels of generality~\cite{CG03,Toft04,wong02}. We will use the
following boundedness result from ~\cite{CG06,Toft07}.   

\par

\begin{lemma}\label{l:1}
Let $m\in \cM _v $ and $\mu \in \cM _w$.   Fix $g \in M^1_{vw}(\rd )
$.  Then the localization operator $A_{1/m}^g $ is bounded from
  $M^{p,q}_\mu (\rd )$ to $M^{p,q}_{\mu m} (\rr d)$. 
\end{lemma}

\par

\rem\ The condition on the window $g$ is required to make sense of
$V_g f$ for $f$ in the domain space $M^{p,q}_\mu (\rr d) $ and of $V_gk$ for
$k$ in the dual $M^{p',q'}_{1/(\mu m)}(\rd )=
(M^{p,q}_{\mu m })'$  of the target space $M^{p,q}_{\mu m}$ for the full
range of parameters $p,q \in [1,\infty ]$. For fixed $p,q\in [1,\infty
]$ weaker conditions  may suffice, because  the norm equivalence
\eqref{eq:18} still holds after relaxing  the condition $g\in M^1_v$  into
$g\in M^{r}_v$ for $r\le \min (p,p',q,q')$~\cite{Toft08}. 

\medspace

On a special pair of \modsp s,  $A_{1/m}^g $ is even an isomorphism~\cite{GT10}. 

\par

\begin{lemma}\label{l2}
  Let $g \in \mvv $, $m\in \cM _v$, and set $\theta =
  m^{1/2}$. Then $A_m^g $ is an isomorphism from $M^{2}_{\theta}(\rd )
  $ onto $M^{2}_{1/\theta }(\rd )$.

Likewise $A_{1/m}$ is an isomorphism from $M^{2}_{1/\theta}(\rd )
  $ onto $M^{2}_{\theta }(\rd )$.   Consequently the composition $A^g_{1/m}
  A^g_m$ is an isomorphism on $M^{2}_\theta (\rd )$,  and  $  A^g_m A^g_{1/m}
 $ is an isomorphism on $M^{2}_{1/\theta } (\rd )$. 
\end{lemma}

\par

Lemma \ref{l2} is based on the equivalence 
\begin{equation}
  \label{eq:3}
\langle A_m^g f ,f \rangle = \langle m, |V_g f|^2 \rangle  =
\|V_g f \cdot \theta \|_2^2 \asymp \|f\|_{M^2_\theta }^2 \, , 
\end{equation}
and is proved in detail in ~\cite[Lemma~3.4]{GT10}. 

\par

\subsection{The Symbol of $A_{1/m}^g A_m^g $}
The composition of localization operators is no longer a
localization operator, but the product of two localization operators still
has a well behaved Weyl symbol. In the following we use the \tf\
calculus of \psdo s as developed in~\cite{gro06,GR08}. Compared to the standard
\psdo\ calculus it is more restrictive because it is related to
the constant Euclidean geometry on phase space, on the other hand,  it is
more general because it works for arbitrary moderate weight functions
(excluding exponential growth). 

\par

Given a symbol $\sigma (x,\xi )$ on $\rr d\times \rr d\simeq \rdd $, the corresponding \psdo\ in the
Weyl calculus $\mathrm{Op}(\sigma )$ is defined formally as 
$$
\mathrm{Op}(\sigma )f(x) = \iint _{\rr {2d}} \sigma \Big (\frac{x+y}{2},\xi \Big ) e^{2\pi i
  (x-y)\cdot \xi } f(y)\, dy d\xi
$$
with a suitable interpretation of the integral. 
If $\cC $ is a class of symbols, we write $\mathrm{Op}(\cC ) = \{
\mathrm{Op}(\sigma ) : \sigma \in \cC \}$ for the class of all  \psdo
s with symbols in $\cC $. 
For the control of the symbol of composite operators we will use the
following characterization of the generalized Sj\"ostrand class from
~\cite{gro06}. For the formulation associate to a submultiplicative
weight $v(x,\xi )$ on $\rdd $ the  rotated weight on $\bR ^{4d}$
defined by 
\begin{equation}\label{vtilde}
\tilde v (x,\xi ,\eta ,y) = v(-y , \eta ).
\end{equation}
For radial weights, to which we will restrict
later, the distinction between $v$ and $\tilde v$ is unnecessary. 

\par

\begin{tm}\label{p3}
  Fix a non-zero $g\in \mvv $. An operator $T$ possesses a Weyl
  symbol in $\mif _{\tilde v}$, $T\in \mathrm{Op}(\mif _{\tilde v})$, \fif\ there exists
  a (semi-continuous) function $H\in L^1_v (\rdd )$ such that 
$$
|\langle T\pi (z)g, \pi (y)g\rangle | \leq  H(y-z) \qquad \text{ for
  all } y,z \in \rdd \, .
$$
\end{tm}

\par

\rem\ 
This theorem says the symbol class $\mif _{\tilde v}$ is \emph{characterized}
by the off-diagonal decay of its kernel with respect to \tfs s. This
kernel is in fact dominated by a convolution kernel. The composition of
operators can then studied with the help of convolution
relations. Clearly this is significantly easier than the standard
approaches that work with the Weyl symbol directly and the twisted
product between Weyl symbols. See~\cite{HoToWa} for results in this
direction. 

\par

\begin{tm}
  \label{t41}
Assume that $g\in M^1_{v^s} (\rd )$, $T\in \mathrm{Op}(\mif _{\tilde v^s})$ for $s\geq 1/2$, and $\theta 
\in \cM _{v^{1/2}}$. Then $A^g_\theta T A^g_{1/\theta } \in \mathrm{Op}(\mif
_{\tilde v^{s-1/2}})$.  
\end{tm}

\begin{proof}
We distinguish the window $g$ of the localization operator $A_m^g$
from the  window ${h} $ used in the expression of the  kernel $\langle
T\pi (z){h} , \pi (y){h} \rangle$. Choose ${h} $ to be the Gaussian,
then ${h} \in \mvv $ for every submultiplicative weight $v$. 
Let us first write the kernel $\langle
T\pi (z){h} , \pi (y){h} \rangle$ informally and justify the
convergence of the integrals later. 
Recall that
$$
T(A_{1/\theta}^g f )  = T\Big( \intrdd \theta (u)\inv
\langle f, \pi (u)g\rangle \pi (u)g \, du  \Big) 
$$
Then 
\begin{multline}\label{eq:21}
\lefteqn{\langle A_\theta ^g T A^g_{1/\theta} \pi (z){h} , \pi (y){h}
  \rangle 
= \langle    T A^g_{1/\theta} \pi (z){h} , A_\theta ^g\pi (y){h}
\rangle  }
\\[1ex]
=\iint _{\rr {4d}} \frac{1}{\theta (u)} \big\langle \pi (z) {h} , \pi
(u)g\big\rangle \, \big\langle T\pi (u)g, \pi (u')g\big\rangle \, \theta (u')
\big\langle \pi (y){h} , \pi (u')g \big\rangle \, dudu'  \,.
  \end{multline}


\par

Now set 
$$
G(z) = |\langle g, \pi (z){h} \rangle | = |V_{{h} } g(z)| \,\,
\text{ and } \,\, G^*(z) = G(-z)
$$
and let $H$ be a dominating function in $L^1_v(\rdd )$, so that
$|\langle T\pi (u)g, \pi (u')g\rangle | \leq H(u'-u)$. 
Since \tfs s commute up to a phase factor, we have 
$$
|\langle \pi (z) {h} , \pi
(u)g\rangle | = G(z-u) \, .
$$
 Before substituting all estimates into \eqref{eq:21}, we recall that
 $\theta $ is $\sqrt{v}$-moderate by assumption and so \eqref{eq:25}
 says that 
$$
\frac{\theta (u')}{\theta (u)} \leq v(u'-u)^{1/2}  \qquad \text{ for all }
u,u'\in \rdd \, .
$$
Now by \eqref{eq:21} we get
\begin{eqnarray}
 \lefteqn{ |\langle A_\theta ^g T A^g_{1/\theta} \pi (z){h} , \pi
   (y){h} \rangle |} \notag
   \\
  &\leq & \intrdd \intrdd \frac{\theta (u')}{\theta (u)} G(z-u)
  H(u'-u) G(y-u') \, dudu' \notag
  \\
&\leq & \intrdd \intrdd  G(z-u) v(u'-u)^{1/2}  H(u'-u) G(y-u') \,
dudu' \notag \\
&=&
\Big( G \ast (v^{1/2}H) \ast G^* \Big)(z-y) \, . \notag
\end{eqnarray}
Thus the kernel of $A_\theta ^g T A^g_{1/\theta}$ is dominated by the
function $G\ast (v^{1/2}H) \ast G^* $. 
By assumption $g\in M^1_{v^s} (\rd )$ and thus $G\in L^1_{v^s}(\rdd
)$,  and $T\in
\mathrm{Op}(\mif _{\tilde v^{s}})$ and thus $H\in L^1_{v^s}(\rdd )$. Then
$v^{1/2}H \in L^1_{v^{s-1/2}}(\rdd )$. 
Consequently
\begin{equation}
  \label{eq:cnv}
G \ast (v^{1/2}H) \ast G^* \in L^1_{v^s} \ast L^1_{v^{s-1/2}}
\ast L^1_{v^s} \subseteq L^1_{v^{s-1/2}}\, .  
\end{equation}

\par

The characterization of  Theorem~\ref{p3}  now implies that $A_\theta ^g
T A^g_{1/\theta} \in \mathrm{Op}(\mif _{\tilde v^{s-1/2}})$. 
\end{proof}

\par

\begin{cor}\label{p4}
Assume that $g \in M^1_{v^2 w}(\rd ) $ and $m\in \cM _v$ and $w$ is an
arbitrary submultiplicative weight. Then 
  $A_{1/m}^g A_m^g \in \mathrm{Op}(\mif _{\tilde {v} \tilde  w}))$. 
\end{cor}

\par

\begin{proof}
In this case $T$ is the identity operator and $\mathrm{Id} \in
\mathrm{Op}(\mif _{v_0})$ for every submultiplicative weight
$v_0(x,\xi ,\eta ,y)=v_0 (\eta ,y)$. In
particular $\mathrm{Id} \in
\mathrm{Op}(\mif _{\widetilde{v}^2})$. Now  replace the weight $\theta $ in
Theorem~\ref{t41}  by $m$
and  the condition $\theta  \in \cM _{v^{1/2}} $ by $m \in \cM _{v}$ and
modify the convolution inequality~\eqref{eq:cnv} in the proof of  Theorem~\ref{t41}. 
\end{proof}


\par

\section{Canonical  Isomorphisms
between Modulation Spaces of Hilbert-Type}

\par

In~\cite{GT10} we have used a deep result of Bony and Chemin~\cite{BC94} about
the existence of isomorphisms between \modsp s of Hilbert type and
then extended those isomorphisms to arbitrary \modsp s. Unfortunately
the result of  Bony and Chemin is restricted to weights of polynomial
type and does not cover weights moderated by superfast growing functions, such as $v(z)
= e^{a|z|^b}$ for $0< b <1$. 

\par

In this section we construct explicit isomorphisms between $\lrd $ and
the  \modsp s $M^2_\theta (\rd )$ for a  general class of
weights. We will assume that the weights are radial in each
\tf\ variable. Precisely, consider \tf\ variables
\begin{equation}\label{R2nCnident}
(x,\xi ) \simeq z= x+i\xi \in \cd \simeq  \rdd ,
\end{equation}
which we identify by
$$
(x_1, \xi _1; x_2, \xi _2; \dots ; x_d, \xi _d) = (z_1, z_2, \dots
, z_d)          \in \cd \simeq  \rdd .
$$
Then the weight function $m$ should satisfy
\begin{equation}\label{mradial}
m(z) = m_0(|z_1|, \dots , |z_d|) \qquad \text{ for } z\in \rdd \, 
\end{equation}
for some function $m_0 $ on $\overline {\bR} _+^d=[0,\infty )^d$. 
 Without loss of generality, we may also  assume that $m$ is
continuous on $\rdd $. (Recall that only weights of polynomial type
occur in the lifting results in \cite{GT10}. On the other hand,  no
radial symmetry   is needed in \cite{GT10}.)

\par

For each multi-index $\alpha =(\alpha _1, \dots ,\alpha _d)
\in \bN _0^d$ we denote the corresponding multivariate Hermite function by 
$$
h_\alpha (t) = \prod _{j=1} ^d h_{\alpha _j} (t_j),\quad \text{where}\quad h_n (x) =
\frac{2^{1/4}\pi ^{n/2}}{n!^{1/2}} e^{\pi x^2} \frac{d^n}{dx^n} (e^{-2\pi x^2})
$$
is the $n$-th Hermite
function in one variable with the normalization $\|h_n\|_{2}=1$. 
%
%
%
%
%
%
%
%
%
%
%

Then the collection of all Hermite functions $h_\alpha , \alpha \geq
0$, is an \onb\ of $\lrd $. By identifying $\rdd $ with $\cd $ via \eqref{R2nCnident},
the \stft\ of $h_\alpha $ with respect to ${h}
(t) = 2^{d/4} e^{-\pi t^2} $ is simply
\begin{equation}
  \label{eq:can1}
  V_{h} h_\alpha (\overline{z}) =  e^{-\pi i x\cdot \xi } \Big(\frac{\pi ^{|\alpha|}}{\alpha
      !}\Big)^{1/2} z^\alpha \, e^{-\pi |z|^2/2} = e^{-\pi i x\cdot
      \xi } \, e_\alpha (z) \,
    e^{-\pi |z|^2}  \qquad \text{ for }
    z\in \cd \, . 
\end{equation}

\par

\rem\ We mention that  a  formal   Hermite expansion $f= \sum _\alpha
c_\alpha h_\alpha $ defines a distribution in the Gelfand-Shilov space
$(S^{1/2}_{1/2})'$, \fif\  the  coefficients satisfy  $|c_\alpha | = \cO (e^{\epsilon
  |\alpha|})$ for every $\epsilon >0$. The Hermite expansion then  converges in the
weak$^*$ topology. Here we have used the fact that the duality $(S^{1/2}_{1/2})' \times S^{1/2}_{1/2}$ extends the $L^2$-form $\langle \cdot , \cdot \rangle$ on $S^{1/2}_{1/2}$, and likewise the duality of the \modsp s
$M^2_\theta (\rd )\times M^2_{1/\theta}(\rd )$.   Consequently, the  coefficient
$c_\alpha $ of a Hermite expansion  is uniquely determined by
$c_\alpha = \langle f, h_\alpha \rangle$ for $f\in (S^{1/2}_{1/2})'$.
See~\cite{janssen82} for details. 

 In the following we take the existence and convergence of    Hermite
 expansions for functions and distributions in arbitrary \modsp s for granted. 
By  $d\mu (z) = e^{-\pi |z|^2} \, dz $  we denote  the Gaussian measure on $\cd $.
\par

\begin{lemma}\label{l6}
Assume that  $\theta(z)  = \cO (e^{a|z|})$ and that  $\theta$ is radial in
each coordinate. 
\begin{itemize}
\item[\rm{(a)}]    Then the monomials $z^\alpha ,\alpha \geq 0$, are orthogonal in
  $L^2_\theta (\cd , \mu )$. 

\item[\rm{(b)}]  The finite linear combinations of the  Hermite functions are dense in $M^2_\theta (\rd )$.
\end{itemize}

\end{lemma}

\par

By using polar coordinates $z_j=r_je^{i\fy _j}$, where $r_j\ge 0$ and $\fy _j\in [0,2\pi )$, we get
\begin{equation}\label{conseqpolar}
z^\alpha =r^\alpha e^{i \alpha \cdot  \fy}\quad \text{and}\quad dz = r_1\cdots r_d\, d\fy dr
\end{equation}
and the condition on $\theta$ in Lemma \ref{l6} can be recast as 
\begin{equation}\label{conseqpolar2}
\theta (z)=\theta _0(r),
\end{equation}
for some appropriate function $\theta _0$ on $[0,\infty )^d$,  and
$r=(r_1,\dots ,r_d)$ and $\fy =(\fy _1,\dots ,\fy _d)$ as usual. 

\par

\begin{proof}
(a)   This is  well-known and is proved  in
\cite{daub88,folland89}. In order to be self-contained,  we  recall
the arguments. 
By writing the integral over $\rdd $ in polar coordinates in
each \tf\ 
pair, \eqref{conseqpolar} and \eqref{conseqpolar2} give 
\begin{equation*}
\intrdd z^\alpha \overline{z^\beta} \theta (z) ^2 \, e^{-\pi |z|^2} \, dz
=  \iint _{\rd _+ \times [0,2\pi   )^d}  e^{i  (\alpha  - 
  \beta ) \cdot  \vf } r^{\alpha +\beta}  e^{-\pi r^2} \,
\theta _0(r)^2 r_1 \cdots r_d\, d\vf  dr \, .
\end{equation*}
The integral over the angles $\vf _j$ is zero, unless $\alpha = \beta
$, whence the orthogonality of the monomials.  

\par

(b) Density: Assume on the contrary that the closed subspace in
$M^2_\theta (\rd )$ spanned by the Hermite functions is a proper subspace of
$M^2_\theta (\rd )$. Then there exists a non-zero  $ f\in (M^2_\theta
(\rd ))' = M^2_{1/\theta} (\rd )$, such that 
$\langle f , h_\alpha   \rangle = 0$ for all Hermite functions
$h_\alpha \in M^2_{\theta }(\rd ) $,  
$\alpha \in \bN _0^d$. Consequently the Hermite expansion of $f =
\sum _\alpha \langle f, h_\alpha \rangle h_\alpha =0$ in
$(S^{1/2}_{1/2})'$, which contradicts the assumption that $f\neq 0$.
\end{proof}
 
 \par
 
\begin{definition}
  The canonical localization  operator $J_m$  is the localization operator $A^{h}
  _m$  associated to the weight $m$ and  to the Gaussian window
  ${h} = h_0$.
  Specifically,
  \begin{equation}
    \label{eq:can2}
    J_m f = \intrdd m (z) \langle f, \pi (z){h} \rangle \pi
    (z) {h} \, dz \, .
  \end{equation}
\end{definition}
 For $m=\theta ^2$ we obtain
\begin{equation}
  \label{eq:can4}
  \langle J_mf,f\rangle = \langle m V_{h} f , V_{h}
  f\rangle _{\rdd } = \| V_{h} f \, \theta \|_2^2 = \|f\|_{M^2_\theta
  }^2 \, 
\end{equation}
whenever $f$ is in a suitable space of test functions. 

\par

Our main insight is that localization  operators with respect to Gaussian
windows and radial symbols have rather special properties. In view of
the connection to the localization  operators on the Bargmann-Fock space
(see below) this is to be expected. 


\begin{tm}\label{t7}
  If $\theta $ is a continuous,  moderate function and radial in each \tf\ coordinate,
  then each of the mappings
\begin{alignat*}{4}
J_{\theta} \, &: &\, M^2_\theta (\rd )&\to L^2(\rd ),&\quad J_{\theta} \, &: &\, L^2(\rd ) &\to M^2_{1/\theta} (\rd )
\\[1ex]
J_{1/\theta} \, &: &\, M^2_{1/\theta} (\rd )&\to L^2(\rd ),&\quad J_{1/\theta} \, &: &\, L^2(\rd ) &\to M^2_{\theta} (\rd )
\end{alignat*}
is an  isomorphism.
\end{tm}

\par

The proof is non-trivial and requires a number of preliminary
results. 
In these investigations we will play with different coefficients of the form
$$
\tau _\alpha (\theta ) := \scal {J_\theta h_\alpha }{h_\alpha },
$$
or, more generally,
\begin{equation}\label{eq:can17}
\tau _{\alpha ,s}(\theta ) := \tau _{\alpha}(\theta ^s) = \scal {J_{\theta ^s}h_\alpha }{h_\alpha } =\intrdd \theta (z)^s \frac{\pi ^{|\alpha|}}{\alpha !} |z^\alpha |^2
  e^{-\pi |z|^2}\, dz \, ,
\end{equation}
when $\theta$ is a weight function and $s\in \mathbb R$. We note that
the $\tau _{\alpha, s}(\theta)$ are strictly positive, since $\theta
$ is positive. If $\theta \equiv 1$, then $\tau _{\alpha ,s} (\theta ) = 1$, so
we may consider the coefficients $\tau _{\alpha ,s}(\theta )$ 
as \emph{weighted gamma functions}. 

\par

\begin{prop}[Characterization of $M^2_\theta$ with Hermite functions]
\label{p8}
  Let $\theta $ be a moderate and radial  function.
Then
\begin{equation}
  \label{eq:can6}
  \|f\|_{M^2_\theta }^2 = \sum _{\alpha \geq 0} |\langle f, h_\alpha
  \rangle | ^2 \tau _\alpha (\theta ^2) \, .
\end{equation}
\end{prop}

\par

\begin{proof}
Let $f= \sum _{\alpha \geq 0} c_\alpha h_\alpha $ be a finite linear
combination of Hermite functions. Since the \stft\ of
$f$ with respect to the Gaussian ${h} $ is given by 
$$
V_{h} f (\overline{z})  = \sum  c_\alpha  V_{h} h_\alpha (\overline{z})
 = e^{-\pi ix\cdot \xi }\sum _{\alpha \geq 0}  c_\alpha  e_\alpha (z) e^{-\pi |z|^2/2}
$$
in view of \eqref{eq:can1}, definition \eqref{eq:can2} gives
\begin{eqnarray*}
  \|f\|_{M^2_\theta }^2 &=& \intrdd |V_{h} f (z) |^2 \theta (z)^2 \,
  dz \\
&=& \sum _{\alpha,\beta \geq 0} c_\alpha \overline{c_\beta } \intrdd
e_\alpha (z) \overline{e_\beta (z)} \theta (z) ^2 \, e^{-\pi |z|^2} \, dz = \sum _{\alpha \geq 0} |c_\alpha |^2 \tau _\alpha (\theta ^2)
 \, .
\end{eqnarray*}
In the latter equalities it is essential that the
weight $\theta $ is radial in each \tf\ coordinate so that the
monomials $e_\alpha $ are orthogonal in $L^2_\theta (\cd , \mu )$. 
\end{proof}

\par

In the next proposition, which is due to Daubechies~\cite{daub88},
we represent the canonical localization  operator 
by a Hermite expansion. 

\par

\begin{prop}\label{p9}
  Let $\theta $ be a moderate, continuous weight function on $\rdd $
  that is radial in each \tf\ coordinate.

\par

 Then the  Hermite function $h_\alpha $ is an  eigenfunction of the localization
operator $J_\theta $ with  eigenvalue $\tau _\alpha (\theta )$ for $\alpha \in \bN _0^d$, and $J_\theta $ possesses the
  eigenfunction expansion 
  \begin{equation}
    \label{eq:can7}
    J_{\theta  } f = \sum _{\alpha \geq 0} \tau _\alpha (\theta )\langle f,
    h_\alpha \rangle  h_\alpha \qquad \text{ for all } f\in \lrd \, .
  \end{equation}
\end{prop}

\par

\begin{proof}
By Lemma~\ref{l6}(a) we find that,   for $\alpha \neq \beta $,   
\begin{eqnarray*}
 \langle J_{\theta}
h_\beta , h_\alpha \rangle &=& \int _{\cd } \theta (z) V_{h}
h_\beta (z) \,  \overline{V_{h} h_\alpha (z)}  \, dz \\
&=& \int _{\cd } \theta (z) e_\beta (z) \overline{e_\alpha (z)}
e^{-\pi |z|^2} \,  dz = 0\, .
\end{eqnarray*}
This implies that $J_\theta h_\alpha = ch_\alpha $ and therefore $c =
c \langle h_\alpha , h_\alpha \rangle = \langle J_\theta h_\alpha ,
h_\alpha \rangle = \tau _\alpha (\theta )$. 

\par

For a (finite) linear combination   $f = \sum _{\beta \geq 0} c_\beta h_\beta $, 
 we obtain 
\begin{multline*}
  J_\theta  f = \sum _{\alpha \geq 0} \langle J_\theta f , h_\alpha \rangle
  h_\alpha 
= \sum _{\alpha \geq 0} \sum _{\beta \geq 0} c_\beta \langle J_\theta
h_\beta  , h_\alpha \rangle h_\alpha
\\[1ex]
=  \sum _{\alpha \geq 0}\sum _{\beta \geq 0} \tau _\beta (\theta )  \delta
_{\alpha ,\beta } c_\beta   h_\alpha 
= \sum _{\alpha \geq 0} \tau _\alpha (\theta )  c_\alpha h_\alpha \, .
\end{multline*}
The proposition follows  because the Hermite functions span $M^2_{1/\theta }(\rd )$ and 
because  the coefficients of a Hermite expansion are  unique and given by
$c_\alpha = \langle f, h_\alpha )$. 
\end{proof}

\par

\begin{cor}\label{c10}
If $\theta $ is moderate and radial in each coordinate, then   $J_{\theta }: \lrd \to
M^2_{1/\theta} (\rd )$ is one-to-one  and
possesses dense range in $M_{1/\theta} ^2(\rd )$.
\end{cor}

\par

\begin{proof}
  The coefficients in $J_\theta  f= \sum _{\alpha \geq 0} \tau _\alpha (\theta )
  \langle f,   h_\alpha \rangle h_\alpha $ are unique. If $J_\theta f= 0$,
  then $\tau _\alpha (\theta )\langle f, h_\alpha \rangle = 0$, and since $\tau
  _\alpha (\theta )>0$ we obtain $\langle f, h_\alpha \rangle = 0$ and thus
  $f=0$. Clearly the range of $J_\theta $ in $M^2_{1/\theta} (
\rd ) $ contains the  finite linear combinations of Hermite functions, 
and these  are   dense in  $M^2_{1/\theta} (\rd )$ by Proposition~\ref{l6}. 
\end{proof}

\par

To show that $J_{\theta} $  maps $\lrd $  onto $M_{1/\theta} ^2(\rd )$ is
much more  subtle. For this we need a new type of  inequalities valid
for the weighted gamma functions  in \eqref{eq:can17}. 
By Proposition~\ref{p9} the number $\tau _{\alpha ,s}(\theta ) $ is exactly the eigenvalue of the
localization operator $J_{\theta ^s}$ corresponding to the
eigenfunction $h_\alpha $.

\par

\begin{prop}
  \label{p11b}
If $\theta \in \cM _w$ is continuous and radial in each \tf\
coordinate, then the mapping $s \mapsto \tau _{\alpha ,s}(\theta )$ is ``almost
multiplicative''. This means that for every $s,t\in \bR $ there exists
a constant $C= C(s,t)$ such that
\begin{equation}
  \label{eq:can118}
C\inv \leq   \tau _{\alpha ,s}(\theta ) \tau _{\alpha ,t}(\theta ) \tau
  _{\alpha ,-s-t}(\theta ) \leq C \qquad \text{
  for all multi-indices } \alpha  \, .
\end{equation}
\end{prop}

\par

\begin{proof}
The upper bound is easy. By Lemma~\ref{l6} the Hermite function
$h_\alpha $ is a common eigenfunction of $J_{\theta ^s}, J_{\theta
  ^t}$, and $J_{\theta ^{-s-t}}$. Since the operator $J_{\theta ^s} J_{\theta
  ^t}J_{\theta ^{-s-t}}$ is bounded on $\lrd $ by repeated application of
Lemma~\ref{l:1},  we obtain that
\begin{equation}
  \label{eq:24}
\tau _{\alpha ,s}(\theta ) \tau _{\alpha ,t}(\theta )\tau _{\alpha ,-s-t} (\theta )=  \| J_{\theta ^s}J_{\theta
  ^t}J_{\theta ^{-s-t}}h_\alpha \|_{L^2} \leq C \|h_\alpha  \|_{L^2} = C \, 
\end{equation}
for all $\alpha \geq 0$.  The constant $C$ is operator norm of $ J_{\theta ^s}J_{\theta
  ^t}J_{-s-t}$ on $\lrd $. 

For the lower bound 
  we rewrite the definition of $\tau _{\alpha,s}(\theta )$ and make it more
  explicit by using polar coordinates $z_j = r_j e^{i\vf _j}, r_j\geq
  0, \vf _j\in [0,2\pi )$,  in each variable. Then  by assumption
  $\theta (z) = \theta _0(r)$ for some
  continuous moderate function $\theta _0$ on $\bR _+^d$, and  we obtain 
\begin{multline}
  \label{eq:20}
    \tau _{\alpha ,s}(\theta ) =  \intrdd \theta (z)^s \frac{\pi ^{|\alpha|}}{\alpha !} |z^\alpha |^2
  e^{-\pi |z|^2}\, dz
\\[1ex]
= (2\pi )^d \int _{\rd _+} \theta _0 (r)  \frac{\pi ^{|\alpha|}}{\alpha !} r^{2 \alpha  }   e^{-\pi |r|^2} \, r_1 \cdots r_d \,  dr
\\[1ex]
=  (2\pi )^d \int _{0}^\infty \dots \int _{0}^\infty  \theta _0(r_1,
\dots , r_d)   \prod _{j=1}^d
\frac{1}{\alpha _j!}\, (\pi r_j^2) ^{\alpha  _j}   e^{-\pi r_j^2} \, r_1
\cdots r_d \, dr_1
\cdots dr_d 
\\[1ex]
=   \int _{0}^\infty \dots \int _{0}^\infty  \theta _0 \big( \sqrt{u_1/\pi},
\dots , \sqrt{u_d/\pi}\big)   \prod _{j=1}^d
\frac{u_j^{\alpha_j}}{\alpha _j!}\,  e^{-u_j}\, du_1 \cdots du_d \, .
\end{multline}
We focus on a single factor in the integral first. The function
$f_n(x) = x^n e^{-x}/n!$ takes its maximum at $x=n$ and
$$
f_n(n) = \frac{1}{n!}n^n e^{-n} = (2\pi n)^{-1/2}\big ( 1+\cO (n^{-1})\big )
$$
by Stirling's
formula. Furthermore, $f_n$ is almost constant on the interval
$[n-\sqrt{n}/2, n+\sqrt{n}/2]$ of length $\sqrt{n}$. On this interval
the minimum of $f_n $ is taken at one of the endpoints $n\pm
\sqrt{n}/2$, where the value is 
$$
f_n(n\pm \sqrt{n}/2) = \frac{1}{n!} (n\pm \sqrt{n}/2)^n e^{-(n\pm
\sqrt{n}/2)} =  \frac{1}{n!}\frac{n^n}{e^n} \big(1\pm
\frac{1}{2\sqrt{n}}\big)^n e^{\mp \sqrt{n}/2} \, .
$$
Since
$$
\lim _{n\to \infty } \frac {\displaystyle{ (2\pi
    n)^{1/2}}\Big( \frac{n}{e}\Big)^n}{n!} =1\quad \text{and}\quad \lim
_{n\to \infty } \big(1\pm
\frac{1}{2\sqrt{n}}\big)^n e^{\mp \sqrt{n}/2} = e^{-1/8}
$$
by Stirling's formula and straight-forward applications of Taylor's formula, we find that 
\begin{equation}
  \label{eq:19}
  f_n(x) \geq \frac{c}{\sqrt{n}} \qquad \text{ for } x\in
  [n-\sqrt{n}/2, n+\sqrt{n}/2] \text{ and all } n\geq 1 \, .
\end{equation}
For $n=0$ we use the inequality $f_0(x) \geq e^{-1/2}$ for $x\in
[0, 1/2]$. 

\par

Now consider the products of the  $f_n$'s occuring in the integral above. 
For $\alpha = (\alpha_1, \dots , \alpha _d) \in \bN _0^d$ define the
boxes 
$$
C_\alpha = \prod _{j=1}^d \big [ \alpha_j-2^{-1}\sqrt{\alpha_j}\, ,\, 
\alpha_j +2^{-1}\max (\sqrt{\alpha_j},1) \big ] \subseteq \rd \, ,
$$
%
with volume $\mathrm{vol}\, (C_\alpha ) = \prod _{j=1}^d \sqrt{\max
  (2^{-1},\alpha _j)} $.  Consequently, on the box $C_\alpha $ we have  
\begin{equation}
  \label{eq:22}
 \prod _{j=1}^d
\frac{u_j^{\alpha_j}}{\alpha _j!}\,  e^{-u_j}\, \geq C_0 \prod _{j=1}^d
\frac{1}{\sqrt{\max( 2\inv ,\alpha _j)}} = C_0 \big(\mathrm{vol}\, C_\alpha \big) \inv 
\end{equation}
for some constant $C_0>0$ which is independent of $\alpha \in \bN _0^d$. 

\par

Next, to take  into account the coordinate change  in \eqref{eq:20},
we define the box     
$$
D_\alpha = \prod _{j=1}^d \Big [ \frac {\displaystyle{(\alpha_j-2^{-1}\sqrt{\alpha_j})^{1/2}}}{\sqrt \pi}\, ,\, 
\frac {\displaystyle{(\alpha_j +2^{-1}\max (\sqrt{\alpha_j},1))^{1/2}}}{\sqrt \pi} \Big ] \subseteq \rd \, ,
$$
%
Furthermore the  length of each edge of $D_\alpha $ is 
$$
\pi ^{-1/2} \Big(\big(\alpha _j+\sqrt{\alpha_j}/2\big)^{1/2}  -
\big(\alpha _j-\sqrt{\alpha_j}/2\big)^{1/2}\Big) \leq \pi ^{-1/2}
$$
when $\alpha_j \geq 1$ and likewise for $\alpha _j=0$. Consequently, 
\begin{equation}
  \label{eq:21A}
  \text{if} \,\, z_1, z_2 \in D_\alpha , \,\, \text{ then }  z_1-z_2
  \subseteq [-\pi  ^{-1/2}, \pi ^{-1/2}]^d \, .
\end{equation}

\par

After these preparations we start the lower estimate of $\tau _{\alpha
  ,s}(\theta )$. Using \eqref{eq:22} we obtain
\begin{eqnarray*}
  \tau _{\alpha ,s}(\theta ) &= &  \int _0^\infty \dots \int _0^\infty    \theta
  _0 \big( \sqrt{u_1/\pi},
\dots , \sqrt{u_d/\pi}\big)^s     \prod _{j=1}^d
\frac{u_j^{\alpha_j}}{\alpha _j!}\,  e^{-u_j}\, du_1 \cdots du_d
\\[1ex]
&\geq &C_0 \frac{1}{\mathrm{vol}\, (C_\alpha ) }\int  _{C_\alpha }
\theta _0 \big( \sqrt{u_1/\pi},
\dots , \sqrt{u_d/\pi}\big)^s  \,  du_1 \cdots du_d \, .
\end{eqnarray*}
Since $\theta $ is continuous, the mean value theorem asserts that
there is a point $z = z(\alpha ,s)  = (z_1, z_2, \dots , z_d
) \in C_\alpha $, such that 
$$
\tau _{\alpha ,s}(\theta ) \geq C_0 \, \theta _0 \big( \sqrt{z_1/\pi},
\dots , \sqrt{z_d /\pi}\big)^s \, .
$$
Note that the point with coordinates $\zeta = \zeta (\alpha ,s) = \big( \sqrt{z_1/\pi},
\dots , \sqrt{z_d /\pi}\big)$ is in $D_\alpha $, consequently
$$
\tau _{\alpha ,s}(\theta ) \geq C_0 \, \theta _0 (\zeta (\alpha ,s)) \qquad \text{
  for } \zeta (\alpha ,s) \in D_\alpha \, .
$$
Finally
\begin{equation}
  \label{eq:23}
\tau _{\alpha ,s}(\theta )\,  \tau _{\alpha ,t}(\theta ) \, \tau _{\alpha ,-s-t}(\theta )  \geq 
C_0^3\,  \theta    _0 (\zeta )^s \, \theta    _0 (\zeta ' )^t \,
\theta    _0 (\zeta '')^{-s-t}
  \end{equation}
for points $\zeta , \zeta ' ,\zeta '' \in D_\alpha $. 
Since the weight $\theta $ is a $w$-moderate,  $\theta _0 $ satisfies 
$$
\frac{\theta _0(z_1)}{\theta _0(z_2)}  \geq \frac{1}{w(z_1-z_2)} \qquad
z_1, z_2 \in \rd \, .
$$
Since $\zeta , \zeta ' ,\zeta '' \in D_\alpha $,  the differences
$\zeta -\zeta ''$ and  $\zeta' -\zeta ''$    are in the cube $[-\pi
^{-1/2}, \pi ^{-1/2}]^d$ as
observed in \eqref{eq:21A}.  We conclude the
non-trivial part of this estimate by 
$$
\tau _{\alpha ,s}(\theta ) \tau _{\alpha ,t}(\theta ) \tau _{\alpha ,-s-t}(\theta ) \geq C_0^3
\frac{1}{w(\zeta - \zeta '')^s} \frac{1}{w(\zeta ' -\zeta '')^t } \geq 
C_0^3 \Big( \max _{z\in [-\pi ^{-1/2}, \pi ^{-1/2}]^d} w(z)\Big)^{-s-t} = C \, .
$$
The proof is complete. 
\end{proof}

\par

The next result provides a sort of symbolic calculus for the canonical
localization operators $J_{\theta ^s}$. Although the mapping $s\to
J_{\theta ^s}$ is not  homomorphism from $\bR $ to operators, it is
multiplicative modulo bounded operators. 

\par

\begin{tm}
  Let $\theta $ and $\mu $ be two  moderate, continuous weight functions on $\rdd $
  that are radial in each \tf\ variable. For every $r,s \in \bR $ there
  exists an operator $V_{s,t}$ that is invertible on every $M^2_{\mu}(\rd )$
  such that
$$
J_{\theta ^s} J_{\theta ^t} J_{\theta ^{-s-t}} =  V_{s,t} \, .
$$
  \end{tm}

\par

  \begin{proof}
For $s,t\in \bR $ fixed,     set $\gamma (\alpha ) = \tau _{\alpha ,s}(\theta )
\tau _{\alpha ,t}(\theta )\tau _{\alpha ,-s-t}(\theta ) $ and 
$$
V_{s,t} f = \sum _{\alpha \geq 0} \gamma (\alpha ) \langle f, h_\alpha
\rangle h_\alpha  \, .
$$
Clearly, $  V_{s,t} = J_{\theta ^s} J_{\theta ^t} J_{\theta ^{-s-t}}
$. 
Since $C\inv \leq \gamma (\alpha ) \leq C$ for all $\alpha \geq 0$ by
Proposition~\ref{p11b}, Proposition ~\ref{p8}  implies that  $
V_{s,t}$ is bounded on  every \modsp\
$M_\mu ^2$. Likewise the formal  inverse operator  $V_{s,t}\inv f   = \sum _{\alpha \geq 0} \gamma
(\alpha ) \inv  \langle f, h_\alpha \rangle h_\alpha $ is bounded on $M_\mu ^2 (\rd )$, consequently $V_{s,t}$
is invertible on $M^2_\mu (\rd )$. 
  \end{proof}

We can now finish the proof of Theorem~\ref{t7}. 

\begin{proof}[Proof of Theorem~\ref{t7}]
Choose  $s=1$ and $t=-1$, then $J_\theta \jjj = V_{1,-1} $ is invertible on
$L^2$.  
Similarly, the choice $s=-1, t=1$ yields that $\jjj J_\theta =
V_{-1,1}$ is invertible on $M^2_\theta$. 
The factorization $J_\theta \jjj = V_{1,-1} $ implies that  $\jjj $ is
one-to-one from $L^2$ to $M^2_\theta$ and that  $J_\theta 
$ maps $M_{\theta }^2$ onto $L^2$. 
The factorization $ \jjj J_\theta = V_{-1,1} $ implies that  $J_\theta 
$ is
one-to-one from $M^2_\theta$ to $L^2$ and that  $J_{1/\theta}$  maps $L^2$ onto $M_{\theta }^2$. 

We have proved that $J_\theta $ is an isomorphism from $M^2_\theta$ to
$L^2$ and that $J_{1/\theta}$ is an isomorphism from $L^2$ to 
$M^2_{\theta}$. The other isomorphisms are proved  similarly. 
\end{proof}

\par

\rem\ In dimension $d=1$  the  invertibility of $J_\theta J_{\theta
  \inv}$ follows from the equivalence $\tau _{n,1}(\theta ) \tau
_{n,-1}(\theta ) \asymp 1$, which can 
be expressed as the following inequality for weighted gamma functions: 
\begin{equation}
  \label{eq:can9}
C\inv \leq    \int _0^\infty \theta _0(\sqrt{x/\pi}) \frac{x^n}{n!} e^{-x} dx
  \, \int _0^\infty \frac{1}{\theta _0(\sqrt{x/\pi})}
  \frac{x^n}{n!} e^{-x} dx \leq C
\end{equation}
for all $n\ge 0$. Here $\theta  _0$ is the same as before. It is a curious and fascinating fact   that this inequality implies that the
localization  operator $J_{1/\theta }$ is an isomorphism between  $L^2(\bR
)$ and $M^2_\theta (\bR )$.  

\par

\section{The General  Isomorphism Theorems}

\par

In Theorem~\ref{t13} we will  state the general isomorphism
theorems. The strategy of the proof is similar  to  that of Theorem~3.2
in~\cite{GT10}. The main tools are the theorems about the spectral
invariance of the generalized Sj\"ostrand classes ~\cite{gro06} and the
existence of a canonical isomorphism between $\lrd $ and $M^2_\theta
(\rd )$ established in Theorem~\ref{t7}. 

%
%
%
%

\par

\subsection{Variations on Spectral Invariance}
We first introduce the tools concerning the spectral invariance of
\psdo s. Recall the following results from \cite{gro06}. 

\par

\begin{tm}\label{spec}
Let $v$ be a submultiplicative weight on $\rdd $ such that
\begin{equation}\label{GRScond}
\lim _{n\to  \infty } v(nz)^{1/n} =1\quad \text{for all } \, z\in \rdd \, ,
\end{equation}
and let $\tilde v$ be the same as in \eqref{vtilde}. If $T\in \mathrm{Op}(\mif _{\tilde v})$ and $T$ is invertible on $\lrd $, then
$T\inv \in \mathrm{Op}(\mif _{\tilde v})$.
 
Consequently, $T$ is invertible simultaneously on all \modsp s 
$M^{p,q}_{{\mu}}(\rd )
$ for $1\leq p,q\leq \infty $ and all $\mu \in \cM_v$. 
  \end{tm}

\par

Condition \eqref{GRScond} is usually called the 
Gelfand-Raikov-Shilov (GRS) condition.

\par

We  prove a more  general form of spectral invariance.
Since we have formulated all results about the canonical localization operators $J_\theta $ 
for radial weights only, we will assume from now on that all weights
are radial in each coordinate. In this case
$$
\tilde v(x,\xi ,\eta ,y) = v(-\eta ,y) = v(y,\eta ),
$$
and we do not need the somewhat ugly distinction
between $v$ and $\tilde v$. 

\par

\begin{tm}\label{t12}
  Assume that $v$ satisfies the GRS-condition, $\theta ^2 \in \cM _v$,
  and that both $v$ and $\theta $ are radial  in each
  \tf\ coordinate.

If $T\in \mathrm{Op}(\mif _v)$ and $T$  is invertible on $M^2_\theta (\rd )$, then $T$
  is invertible on $\lrd $. 

As a consequence $T $ is invertible on every \modsp\ $M^{p,q}_\mu (\rd )$ for
$1\leq p,q\leq \infty $ and $\mu \in \cM _v$.  
\end{tm}

\par

\begin{proof}
  Set $\widetilde{T} = J_\theta T \jjj $. By Theorem~\ref{t7}, $\jjj $ is an
  isomorphism from $\lrd $ onto $M^2_\theta (\rd )$ and $J_\theta $ is
  an isomorphism from $M_\theta ^2 (\rd )$ onto $\lrd $, therefore
  $\widetilde{T} $ is an isomorphism on $\lrd $.

\begin{equation} \label{diagram1}
\begin{matrix}
&M^2_\theta   & \stackrel{T}{\longrightarrow} & M^2_\theta & \cr
&\uparrow J_{1/\theta}  & &\downarrow J_\theta & \cr
 &\lrd \,  &\stackrel{\widetilde{T}}{\longrightarrow} &\lrd  &
\end{matrix}
\end{equation}

By Theorem~\ref{t41} the operator $\widetilde{T} $ is in $\mathrm{Op}( \mif
_{v^{1/2}  })$. 
Since   $\widetilde{T}  $ is invertible on 
$\lrd $,  Theorem~\ref{spec} on the spectral invariance of the
symbol class $\mif _{v^{1/2} }$ implies that the inverse operator
$\widetilde{T}$ also possesses a symbol in $ \mif
_{v^{1/2}}$, i.e.,  $\widetilde{T}\inv \in \mathrm{Op}(\mif _{v^{1/2}})$. 

Now, since $\widetilde{T} \inv = J_{1/\theta }\inv T\inv J_\theta  \inv$,
we find that
$$
T\inv = J_{1/\theta} \widetilde{T} \inv J_\theta  \, .
$$
Applying Theorem~\ref{t41} once again, the symbol of $T\inv $ must be in
$\mif $. As a  consequence, $T\inv $ is bounded on $\lrd $. 

Since 
 $T\in \mathrm{Op}(\mif _v
)$ and $T$ is invertible on $\lrd
$, it follows that $T$ is also invertible on $M^{p,q}_\mu (\rd ) $ for every weight
$\mu \in \cM _v$ and $1\leq p,q\leq \infty $. 
\end{proof}

\par

\subsection{An isomorphism theorem for localization operators}

\par

We now combine all steps and formulate and prove our main result, the
isomorphism theorem for \tf\ localization operators with symbols of
superfast growth. 

\begin{tm}\label{t13}
  Let $g\in M^1_{v^2 w}(\rd ) $, $\mu \in \cM _w$ and $m\in \cM _v$ be
  such that $m$ is radial in each time-frequency coordinate and $v$
  satisfies \eqref{GRScond}. Then the 
  localization operator $A^g_m$ is an isomorphism from $M^{p,q}_{\mu }(\rd )$
  onto $M^{p,q}_{\mu /m}(\rd )$ for $1\leq p,q \leq \infty$. 
\end{tm}

\par

\begin{proof}
Set $T=A^g _{1/m} A^g _m$.   We have  already established that 
  \begin{enumerate}
  \item $T=A^g _{1/m} A^g _m$ possesses a
  symbol in $\mif _{\tilde v \tilde w}$ by Corollary~\ref{p4}. 
\item $T$ is invertible on $M^2_\theta (\rd )$ by Lemma~\ref{l2}.  
  \end{enumerate}
These are the assumptions of Theorem~\ref{t12}, and therefore $T$ is
invertible on $M^{p,q}_\mu (\rr d)  $ for every $\mu \in \cM _{w} \subseteq
\cM  _{vw}$ and $1\leq p,q\leq
\infty $. 
The factorization $T=A^g _{1/m} A^g _m$ implies that $A^g_m$ is
one-to-one from $M_{\mu}^{p,q} (\rd )$ to $M^{p,q}_{\mu /m} (\rd )$ and that 
$A^g_{1/m}$  maps $M^{p,q}_{\mu /m} (\rd )$ onto $M^{p,q}_\mu (\rd ) $. 

\par

Now we change the order of the factors and consider the operator $T' = 
A^g _m A^g _{1/m} $ from $M^{p,q}_{\mu / m} (\rd )$ to $M^{p,q}_{\mu / m} (\rd )$
and factoring through $M^{p,q}_{\mu}(\rd )$. Again $T'$ possesses a symbol
in $\mif _{\tilde v \tilde w}$ and is invertible on $M^2_{1/\theta}(\rd )$. With
Theorem~\ref{t12}   we
conclude that $T' $ is invertible on all \modsp s $M^{p,q}_{\mu /m}(\rd )$. 
The factorization of $T' =A^g _m A^g _{1/m} $ now yields that
$A^g_{1/m}$ is one-to-one from $M^{p,q}_{\mu / m} (\rd )$ to  $M^{p,q}_{\mu}(\rd )$
and that $A^g_m$  maps $M^{p,q}_{\mu}(\rd )$ onto  $M^{p,q}_{\mu / m} (\rd )$. 

\par

As a consequence $A_m ^g$ is bijective from $M^{p,q}_{\mu }(\rd )$
  onto $M^{p,q}_{\mu /m}(\rd )$,  and   $A^g_{1/m}$ is    bijective from
  $M^{p,q}_{\mu /m}(\rd )$ onto $M^{p,q}_{\mu }(\rd )$. 
\end{proof}

\par

\section{Consequences for Gabor Multipliers
and Toeplitz Operators on Bargmann Fock Space}

\par

\subsection{Gabor Multipliers}

Gabor multipliers are \tf\ localization operators whose symbols are
discrete measures. Their basic properties are the same, but  the discrete
definition makes them more accessible for numerical computations. 

Let $\Lambda = A\zdd $ for some $A\in \mathrm{GL}(2d,\bR )$ be a
lattice in $\rdd $, $g, \gamma $ suitable window functions and $m$ a
weight sequence defined on $\Lambda$. 
Then the Gabor multiplier $G^{g,\gamma,\Lambda }_m$ is defined to be 
\begin{equation}
  \label{eq:can10}
  G^{g,\gamma,\Lambda }_mf = \sum _{\lambda \in \Lambda } m(\lambda )
  \langle f, \pi (\lambda )g\rangle \pi (\lambda )\gamma \, .
\end{equation}
The boundedness of Gabor multipliers  between \modsp s is formulated
and proved exactly as for \tf\ localization operators. See~\cite{FN03} for a detailed exposition of Gabor
multipliers and ~\cite{CG03} for general boundedness results that
include distributional symbols.  

\begin{prop}\label{p14}
  Let $g ,\gamma \in M^1_{vw} (\rd )$, $m$ be a continuous moderate weight
  function $m\in \cM _v$. Then $  G^{g,\gamma,\Lambda }_m$ is bounded
  from $M_\mu ^{p,q}(\rdd )$ to $M^{p,q}_{\mu /m}(\rd )$. 
\end{prop}

In the following we will assume that the set $\cG (g, \Lambda )= \{
\pi (\lambda )g: \lambda \in \Lambda \}$ is a
Gabor frame for $\lrd $. This means that the frame operator
$S_{g,\Lambda }f = M^{g,g,\Lambda }  _{1} = \sum _{\lambda \in \Lambda } 
  \langle f, \pi (\lambda )g\rangle \pi (\lambda )g  $ is invertible
  on $\lrd $. From the rich theory of Gabor frames 
  we quote only the following result: If $g\in \mvv $ and $\cG (g,
  \Lambda )$ is a frame for $\lrd $, $1\leq p \leq \infty $, and $m\in
  \cM _v$, then 
  \begin{equation}
    \label{eq:can15}
    \|f\|_{M^p_m} \asymp \Big( \sum_{\lambda \in \Lambda } |\langle f,
    \pi (\lambda )g\rangle |^p m(\lambda )^p\Big) ^{1/p} \, \qquad
    \text{ for all } f\in M^p_m (\rd ).
  \end{equation}
 We refer to ~\cite{book} for a detailed discussion,  the
  proof, and  for further references.

Our main result on Gabor multipliers is the  isomorphism theorem. 

\begin{tm}\label{t15}
Assume that   $g  \in M^1_{v^2w} (\rd )$ and that $\cG (g, \Lambda )$ is
a frame for $\lrd $. If  $m\in \cM _v$ is radial in each \tf\
coordinate,  then $  G^{g,g,\Lambda }_m$ is an 
  isomorphism 
  from $M_\mu ^{p,q}(\rdd )$ onto $M^{p,q}_{\mu /m}(\rd )$ for every
  $\mu \in \cM _w$ and $1\leq p,q\leq \infty $.
\end{tm}

The proof is similar to the proof of Theorem~\ref{t13}, and we only
sketch the necessary modifications. In the following
we fix the lattice $\Lambda $ and choose $\gamma = g$ and drop the
reference to these additional parameter by writing $G^{g,\gamma,
  \Lambda } _m $ as $G^g_m$ in analogy to the localization operator $A^g_m$. 

\begin{prop}\label{p16}
  If $g\in M^1_{v}(\rd )$ and $m\in \cM _v$ and $\theta = m^{1/2}$,  then
  $G^g_m$ is an isomorphism 
  from $M^2_\theta (\rd )$ onto $M^2_{1/\theta } (\rd )$.
\end{prop}

\begin{proof}
  By Proposition \ref{p14} $G_m^g $ is bounded from $M^{2}_{\theta}$ to
  $M^{2}_{\theta /m} = M^{2}_{1/\theta}$ and thus
$$
\|G_m^g f \|_{M^{2}_{1/\theta}} \leq C \|f\|_{M^2_\theta} \, .
$$
Next we use the characterization of $M^2_\theta $ by Gabor frames and
relate it to Gabor multipliers: 
\begin{equation}
  \label{eq:3a}
\langle G_m^g f ,f \rangle =  \sum _{\lambda \in \Lambda }
m(\lambda)  |\langle f, \pi (\lambda )g\rangle |^2 \asymp
\|f\|_{M^2_\theta }^2 \, . 
\end{equation}
This identity implies that $G_m ^g $ is one-to-one on $M^2_\theta $
and that $\|G_m^g f \|_{M_{1/\theta}}$ is an equivalent norm on
$M^2_\theta $. Since $G^g_m$ is self-adjoint, it has dense range in
$M^2_{1/\theta }$, whence $G^g_m$ is onto as well.  
 \end{proof}

Using the weight $1/m$ instead of $m$  and $1/\theta $ instead of
$\theta $, we see that $G_{1/m}^g $ is an isomorphism from
$M^2_{1/\theta } (\rd )$ onto $M^2_\theta (\rd )$. 

\begin{proof}[Proof of Theorem~\ref{t15}]
 We proceed as  in the proof of Theorem~\ref{t13}. Define the operators
  $T= G_{1/m}^g G_m^g $ and $T' = G_m^g G_{1/m}^g  $. By Proposition~\ref{p14} $T$    maps
$M^2_\theta (\rd )$ to $M^2_\theta (\rd )$, and since $T$ is a composition of two
isomorphisms, $T$ is an isomorphism on  $M^2_\theta (\rd )$. Likewise $T'$
is an isomorphism on $M^2_{1/\theta}(\rd )$.

As in Corollary~\ref{p4} we verify that the symbol of $T$ is in $\mif
_{\tilde v \tilde w}
(\rdd )$. Theorem~\ref{t12} then  asserts that $T$ is invertible on $\lrd
$ and  the general spectral invariance implies that $T$ is an
isomorphism on $M^{p,q}_\mu (\rd )$. Likewise $T'$ is an
isomorphism on $M^{p,q}_{\mu/m}(\rd )$.  This means that each of the factors of $T$ must be
an isomorphism on the correct space. 
  \end{proof}

\subsection{Toeplitz Operators}

 The Bargmann-Fock space $\cF = \cF ^2(\cd )$ is the Hilbert space of
 all entire functions of $d$ variables  such that
 \begin{equation}
 \label{Fock_definition}
 \|F\|^2_\cF=\int_ \cd |F(z)|^2 e^{-\pi |z|^2} dz <\infty \, .
 \end{equation}
 Here $dz$ is the  Lebesgue measure on $\cd = \rdd $.

Related  to the Bargmann-Fock space $\cF ^2$   are spaces of entire
functions satisfying weighted integrability conditions. 
Let $m$ be a moderate weight on $\cd $ and $1\leq p,q\leq
\infty$. Then the space $\cF ^{p,q}_m (\cd )$   is the Banach space of
all entire functions of $d$ complex variables, such that 
 \begin{equation}
 \label{Fock_definition-b}
 \|F\|_{\cF^{p,q}_m}^q =\int_{\rd } \Big(\intrd  |F(x+iy)|^p m(x+iy)^p
 e^{-p\pi |x+iy|^2/2} dx \Big)^{q/p}  dy <\infty \, .
 \end{equation}

\par

Let $P$ be the orthogonal projection from $L^2(\cd , \mu )$ with
Gaussian measure $d\mu (z)   = e^{-\pi |z|^2} \, dz$ onto $\cF ^2(\cd
)$.   Then  $P$ is  given by
the formula
\begin{equation}
  \label{eq:can11}
  PF(w) = \int _{\cd } F(z) e^{\pi \overline{z}w} e^{-\pi |z|^2} \, dz \, .
\end{equation}
We remark that a  function on $\rdd $ with a certain growth  is entire
\fif\ it satisfies $F = PF$.
 
The classical Toeplitz operator on Bargmann-Fock space with a symbol
$m$  is defined by   $T_m F = P(m F)$ for $F\in \cF ^2(\cd )$. Explicitly
$T_m$ is given by   the formula
\begin{equation}
  \label{eq:can12}
  T_m F(w)  = \int_{\cd }  m(z) F(z) e^{\pi \overline{z}w} \, e^{-\pi |z|^2}  \, dz \, .
\end{equation}
See~\cite{bargmann61,Berezin71,BCo87,BCo94,daub88,folland89} for a
sample of references. 
 
In the following we assume that the symbol of a Toeplitz operator is
continuous and radial in each coordinate, i.e., $m(z_1, \dots , z_d) =
m_0 (|z_1|, \dots |z_d|)$ for some continuous function $m_0$ from $\bR
_+^d$ to $\bR _+$.

\begin{tm}\label{t20}
  Let  $\mu \in \cM _w$, and let $m\in \cM _v$ be a  continuous moderate
  weight function such that one of the following conditions is  fulfilled:
  \begin{enumerate}
  \item Either $m$ is radial in each coordinate, 
  
  \vspace{0.1cm}
  
  \item or  $m$ is of polynomial type.
 \end{enumerate}
  Then the   Toeplitz operator $T_m$ is an isomorphism from $\cF^{p,q}_{\mu }(\cd
  )$   onto $\cF^{p,q}_{\mu /m}(\cd )$ for $1\leq p,q \leq \infty$. 
\end{tm}

The formulation of   Theorem~\ref{t20} looks  similar to the main theorem
about \tf\ localization operators. In fact, after a suitable
translation of concepts,  it is  a special case of Theorem~\ref{t13}.

To explain the connection, we recall the Bargmann transform  that maps distributions
on $\rd $ to entire functions on $\cd $. 
\begin{equation}
  \label{eq:can13b}
   \cB f(z)=F(z)= 2^{d/4}e^{-\pi z^2 /2}\int_{\rd } f(t)e^{-\pi
    t^2}e^{2\pi t \cdot z} dt \qquad    z\in \cd \, . 
\end{equation}
If $f$ is a distribution, then we interpret the integral as the
action of $f$ on the function $e^{-\pi
    t^2}e^{2\pi t\cdot  z}$.

The connection to \tfa\ comes from the fact that the Bargmann
transform is just a \stft\ in disguise~\cite[Ch.~3]{book}.  As before, we use the
normalized Gaussian ${h} (t) = 2^{d/4} e^{-\pi t^2}$ as a
window. Identifying the pair $(x,\xi)\in \rd \times \rd $ with the
complex vector $z= x+i\xi \in \cd$, the \stft\ of a function $f$ on $\rd
$ with respect to ${h}$ is 
\begin{equation}\label{eq:26b}
  V_{{h} } f(\overline{z}) = e^{\pi i x\cdot \xi } \cB f(z) e^{-\pi
      |z|^2/2} \, .
\end{equation}
In particular, the Bargmann transform of the \tfs\ $\pi (w){h} $ is
given by 
\begin{equation}
  \label{eq:27}
  \cB (\pi (w){h} )(z) = e^{-\pi i u\cdot \eta } e^{\pi \overline{w}z}
  e^{-\pi |w|^2/2} \, \quad  w,z\in \cd , w= u+i\eta \, .
\end{equation}
It is a basic fact that the Bargmann transform   is a unitary mapping
from  $L^2(\bR)$ onto the Bargmann-Fock space $\cF ^2(\cd )$. 
Furthermore, the Hermite functions $h_\alpha , \alpha \geq 0$ are
mapped to the normalized monomials $e_\alpha (z) = \pi ^{|\alpha|/2}
(\alpha !)^{-1/2} z^\alpha  $. 

Let $m'(z) = m(\overline{z})$. Then \eqref{eq:26b} implies that the
Bargmann transform $\cB $ maps $\Mmpqd $ isometrically to $\cF
^{p,q}_{m'}(\cd  )$.
%
%
%
%
%
By straight-forward arguments it follows that  the Bargmann transform
maps the \modsp\  $\Mmpqd $ onto the 
Fock space $\cF ^{p,q}_{m'}(\cd )$~\cite{book,signahltoft}. 

The connection between \tf\ localization operators and Toeplitz
operators on the Bargmann-Fock space is given by the following
statement.

\begin{prop} \label{p21}
Let $m$ be a moderate weight function on $\rdd \simeq \cd $ and set
$ m' (z) = m(\overline{z})$. The Bargmann transform intertwines $J_m$
and $T_{ m'}$, i.e., for all $f \in \lrd $ (or $f\in
(S^{1/2}_{1/2})^*$) we have 
\begin{equation}
  \label{eq:28}
  \cB (J_m f) = T_{ m'} (\cB f ) \, .
\end{equation}
\end{prop}

\begin{proof}
  This fact is well-known~\cite{Cob01,daub88}. For completeness
  and consistency of notation, we provide the formal calculation.

We take the \stft\   of  $J_m= A^{h} _m$  with respect to  ${h}
$. On the one hand we obtain that
\begin{equation}
  \label{eq:29}
   \langle  J_m f, \pi (\overline{w}){h} \rangle  =  e^{\pi i u\cdot \eta }
 \cB (J_mf)(w) \, e^{-\pi |w|^2/2} \, ,
\end{equation}
and on the other hand, after substituting~\eqref{eq:26b} and
\eqref{eq:27}, we obtain that 
\begin{eqnarray}
   \langle  J_m f, \pi (\overline{w}){h} \rangle & = &
\intrdd m(\overline{z}) V_{h} f (\overline{z})
 \overline{V_{h} ( \pi (\overline{w})h)(\overline{z})}  \, dz \notag \\
&=&  \int _{\cd } m(\overline{z}) \cB f(z) \,   \overline{\cB (\pi (\overline{w}){h}
  )(\overline{z})} e^{-\pi |z|^2} \, dz    \label{eq:can114} \\
&=& e^{\pi i u\cdot \eta } \int _{\cd } m(\overline{z}) \, \cB f(z) e^{\pi
  \overline{z}w} \, e^{-\pi |z|^2} \, dz \,\, e^{-\pi |w|^2/2} \, . \notag 
\end{eqnarray}
Comparing \eqref{eq:29} and \eqref{eq:can114} we obtain that
$$
\cB (J_mf)(w) = \int _{\cd } m(\overline{z}) \cB f(z) e^{\pi \overline{z}w}\,
e^{-\pi |z|^2} \, dz = T_{ m'} \cB f(w) \, .
$$
\end{proof}

\begin{proof}[Proof of Theorem~\ref{t20}]
First assume that (1) holds, i.e., $m$ is radial in each variable. The Bargmann transform is an isomorphism between the \modsp\ $\Mmpqd $ and
the Bargmann-Fock space $\cF ^{p,q}_{m'}(\cd )$ for arbitrary moderate weight
function $m$ and $1\leq p,q\leq\infty $. By Theorem~\ref{t13}  the canonical
localization operator $J_m$ is an isomorphism from $M^{p,q}_\mu (\rd ) $ onto
$M^{p,q}_{\mu /m} (\rd )$. Since $T_{ m'}$ is a composition of three
isomorphism (Proposition~\ref{p21} and \eqref{diagram2}), the Toeplitz
operator $T_{  m'}$ is an isomorphism from $\cF _{\mu '} ^{p,q}(\cd )$ onto
$\cF _{\mu ' / m'} ^{p,q}(\cd )$. 
\begin{equation} \label{diagram2}
\begin{matrix}
&\cF _{\mu '} ^{p,q}   & \stackrel{T_{ m'}}{\longrightarrow} & \cF
_{\mu '/ m'} ^{p,q} &
\cr
&\uparrow \cB  & &\uparrow \cB &
\cr
 &M^{p,q}_\mu  \,  &\stackrel{J_m}{\longrightarrow} &M^{p,q}_{\mu /m}  &
\end{matrix} \, .
\end{equation}
Finally replace $m'$ and $\mu'$ by $m$ and $\mu$. 

\par

By using Theorem 3.2 in \cite{GT10} instead of Theorem \ref{t13}, the same arguments show that the result follows when (2) is fulfilled.
  \end{proof}

\def\cprime{$'$} \def\cprime{$'$} \def\cprime{$'$} \def\cprime{$'$}
  \def\cprime{$'$}

 \bibliographystyle{abbrv}
 \bibliography{general,new}

\end{document}